\documentclass[12pt]{article}
\usepackage{amsmath,amsfonts,amssymb,amsthm,amscd}
\title{Domination and regularity}
\date{\today}
\author{Anand Pillay\thanks{Supported by NSF grants DMS-1360702 and DMS-1665035}\\University of Notre Dame }
\newtheorem{Theorem}{Theorem}[section]
\newtheorem{Proposition}[Theorem]{Proposition}
\newtheorem{Definition}[Theorem]{Definition} 
\newtheorem{Remark}[Theorem]{Remark}
\newtheorem{Lemma}[Theorem]{Lemma}
\newtheorem{Corollary}[Theorem]{Corollary}
\newtheorem{Fact}[Theorem]{Fact}

\newcommand{\Q}{\mathbb Q}  
\newcommand{\Z}{\mathbb Z}  
\newcommand{\N}{\mathbb N}

\begin{document}
\maketitle

\begin{abstract} 
We discuss the close relationship between structural theorems in (generalized) stability theory, and graph regularity theorems.
\end{abstract}

\section{Introduction and preliminaries}
We point out  analogies between domination theorems in model theory and graph regularity theorems in various ``tame" contexts, showing that these are essentially the same theorems, modulo compactness and the pseudofinite yoga,  and if one is not so concerned with optimal bounds.

The  motivation comes partly from our joint works with Conant and Terry \cite{CPT1}, \cite{CPT2}, where ``tame" regularity theorems in a group environment {\em are} obtained  from  structural theorems (sometimes new) concerning stable and $NIP$ groups.

We will give later precise statements of all theorems (as well as references to other works). But for now we give a heuristic introduction to the notions in this paper.

First on the graph-theoretic side we recall  the regularity theorems which specialize the well known Szemer\'{e}di regularity theorem. We will focus on bi-partite graphs. Szemer\'{e}di regularity concerns {\em all} finite  graphs  $(V,W, E)$. It says that one can partition the vertex sets $V, W$ into a small number of sets $V_{1},..,V_{n}$, $W_{1},..,W_{m}$ such that outside a small exceptional set of pairs $(i,j)$, the induced subgraphs $(V_{i}, W_{j}, E|(V_{i}\times W_{j}))$ are almost regular, namely sufficiently large induced subgraphs have approximately the same density.  These are approximate or asymptotic statements in the sense that for every $\epsilon >0$ there is $N_{\epsilon}$ such that for every finite graph etc.

Tame versions of Szemer\'{e}di regularity place restrictions on the class of finite graphs $(V, W, E)$ considered, and try to get stronger conclusions.    The kind of restrictions are: omitting a certain induced subgraph, being uniformly definable in some nice structure, or being the collection of finite induced subgraphs of some given graph definable in a nice structure.
The improvements in the conclusions typically replace almost regularity by almost homogeneity (and sometimes outright homogeneity so giving a Ramsey-type theorem) and sometimes remove the need for the exceptional set.

On the model theory side, we work with theories $T$, or formulas $\phi(x,y)$ in a given theory, which are well-behaved in various senses, and we consider a Keisler measure $\mu$ on the $x$-sort over a saturated model $\bar M$, possibly restricted to definable sets in the Boolean algebra generated by instances of $\phi(x,y)$.  The domination statements have the form: there is a small model $M_{0}$,  suitable space $S$ of types over $M_{0}$, such that if $\mu_{0}$ is the measure on $S$ induced by $\mu$ then we have (generic) {\em domination} of the $x$ sort $X$ say by $S$ via the tautological map $\pi:X\to S$ taking $a\in X$ to its type over $M_{0}$:  for any suitable formula $\psi(x)$ over $\bar M$, outside a closed subset $E_{\psi}$ of $S$ of $\mu_{0}$-measure $0$, each fibre of $\pi$ cannot meet both $\psi(x)$ and $\neg\psi(x)$ in a ``$\mu$-wide" set.  

This is actually closely related to a stationarity statement: $\mu$ is the unique nonforking extension of its restriction to $M_{0}$. And we also see an exceptional set $E$ appearing, as in the graph regularity statement.

The work  with Conant and Terry mentioned earlier is  concerned with regularity (and structure) theorems in the context of finite groups $G$ equipped with a distinguished subset $A$.  These give rise to bipartitite graphs of the form $(G, G, E)$ where  $(a,b)\in E$ iff $ab\in A$.  Under assumptions ($k$-stable, $k$-$NIP$) on the relation $E$, we obtained  strong theorems on the structure of the set $A$ and its translates, where local stable and $NIP$ group theory played  a major role. We refer the reader to the preprints \cite{CPT1}, \cite{CPT2} and we will not explicitly discuss these group results any further in the current paper. 

We will go through three model-theoretic situations where there is a domination statement; smooth measures, generically stable measures in $NIP$ theories, and $\phi$-measures where $\phi(x,y)$ is stable.  In each environment we will conclude more or less directly, via compactness,  the relevant graph-regularity statement for suitable classes of finite graphs.   These statements are already in the literature in various forms and we will give full references.

This paper is  based partly on  seminar talks the author gave at the Institut Henri Poincar\'{e} in spring 2018 during the trimester on Model theory, Combinatorics and Valued Fields. Thanks to the IHP for its hospitality and to the organizers of the trimester and the seminars. Thanks  to Gabriel Conant and Caroline Terry for many discussions.  And for the record I would also  like to thank Udi Hrushovski who already in 2012 pointed out to me (and our co-athor) connections between the Lovasz-Szegedy paper \cite{L-Sz} and our paper  \cite{NIPIII} (in particular generic compact domination).

\vspace{5mm}
\noindent
No additional background is needed on the combinatorial side, as all the relevant statements (rather than proofs) are transparent.  

On the model theory side we will make use of Keisler measures in a $NIP$ and (formula-by-formula) stable environment. But we will make precise a few things which are not made explicit in the literature although should be considered folklore,

 Our model theory notation is standard. $T$ denotes a complete theory in a language $L$ and we will work in a very saturated or monster model $\bar M$ of $T$.

The book  \cite{Simon-book} is a useful reference for material on the $NIP$ side, but we will usually refer to the original sources  \cite{NIPI}, \cite{NIPII} for Keisler measures, and \cite{NIPIII} for generically stable and smooth measures. 
Insofar as stability is concerned \cite{Pillay-book} is a reference, although we take our definition of forking to be Shelah's. 

For $\phi(x,y)$ an $L$-formula, by a $\phi$-measure $\mu$ over $M$ we mean a finitely additive probability measure on the Boolean algebra of $\phi$-formulas over $M$, where by a $\phi$-formula over $M$ we mean a (finite) Boolean combination of instances $\phi(x,b)$ of $\phi(x,y)$ with $b\in M$. ``Global" means over the monster model. 

As usual a $\phi$-measure over $M$ can be identified with a regular Borel probability measure on the space $S_{\phi}(M)$ of complete $\phi$-types over $M$. 

  A $\phi$-measure over $M$ is said to be smooth if it has a unique extension to a $\phi$-measure over any larger model.

A characteristic property of $\phi$-measures when $\phi(x,y)$ is stable is the following (see also  Lemma 1.7 of \cite{Keisler}):
\begin{Fact} Let $\phi(x,y)$ be stable (for $T$). Then any $\phi$-measure, $\mu_{x}$, over a model $M$ is of the form 
$\sum_{i = 1, 2,...}\alpha_{i}p_{i}$ where $p_{i}$ is a complete $\phi$-type over $M$, the $\alpha_{i}$ are positive real numbers, and $\sum_{i}\alpha_{i} = 1$.
\end{Fact}
\begin{proof} We give a proof, for completeness, as this has not been made so explicit in  earlier papers.  We use Shelah's  $\phi$-rank $R_{\phi}(-)$ from Section 3, Chapter 1, of  \cite{Pillay-book} where its basic properties are given (and where really we mean  $\Delta$-rank where $\Delta = \{\phi(x,y), x=z\}$). 
Let $p_{1},.., p_{k}$ be the finitely many complete $\phi$-types of maximal $\phi$-rank $n$ say.  Without loss of generality $p_{1},..,p_{r}$ have positive $\mu$-measures, (say $\alpha_{1},..,\alpha_{r}$, respectively) and $p_{r+1},..,p_{k}$ have $\mu$-measure $0$.  

Working in the space $S_{\phi}(M)$ let $U$ be the complement of $\{p_{1},..,p_{r}\}$, an open set whose $\mu_{x}$-measure is  $\beta = 1- (\alpha_{1} + .. + \alpha_{r})$, which we can assume to be positive  (otherwise already $\mu = \alpha_{1}p_{1} + ... + \alpha_{r}p_{r}$).  Now we can find clopen $U_{1}\subset U_{2} ...\subset U_{i} \subset ....  \subset U$, and 
positive reals $\beta _{1} < \beta _{2} < ....<\beta_{i} < ....$ such that   $\mu(U_{i}) = \beta_{i}$ for all $i$ and $lim_{i\to\infty}\beta_{i} = \beta$. 

Now $U_{1}$ and each $U_{i+1}\setminus U_{i}$ are $\phi$-definable sets of positive measure and with $\phi$-rank $<n$.  So we can apply induction, to write 
each of $\mu|U_{1}$, ..., $\mu|(U_{i+1}\setminus U_{i})$,....  as a suitable $\sum_{j}\gamma_{j}q_{j}$. Putting these together with $\alpha_{1}p_{1} + ... + \alpha_{r}p_{r}$ gives the required expression of $\mu$. 
\end{proof} 

The following is not required, but included for completeness. 
\begin{Corollary} If $\phi(x,y)$ is stable and $\mu_{x}$ is a $\phi$-measure over $M$. Then $\mu$ is smooth if and only if $\mu$ is a weighted sum of realized $\phi$-types, i.e. of the form $\sum_{i}\alpha_{i}tp_{\phi}(a_{i}/M)$  with $a_{i}\in M$.
\end{Corollary}

Finally we discuss pseudofiniteness.

\begin{Definition} Let $M$ be an $L$-structure and $A$ an arbitrary (not necessarily definable) subset of a sort $X$ in $M$. We say that $A$ is pseudofinite in $M$ if for any sentence $\sigma$ in the language  $L$ together with an additional preciate symbol $P$ on sort $X$, if $(M,A)\models \sigma$ then there is an $L$-structure $M'$ and finite subset $A'$ of $X(M')$ such that $(M',A')\models \sigma$. 

If $M$ is $1$-sorted and $A$ is $M$ itself then we say that $M$ is pseudofinite.

\end{Definition}

From the definition, finite implies pseudofinite.

\begin{Remark} Suppose that $A$ happens to be definable by a formula $\phi(x,b)$ in the structure $M$. Then pseudofiniteness of $A$ in $M$ is equivalent to; for every $L$-formula $\psi(y)\in tp_{M}(b)$ there is an $L$-structure $M'$ and $b'\in M'$ satisfying $\psi(y)$ such that $\phi(x,b')(M')$ is finite.
\end{Remark}

The following is routine.
\begin{Lemma} For $M$ an $L$-structure and $A$ a subset of a sort $X$ in $M$, the following are equivalent:
\newline
(i) $A$ is pseudofinite in $M$,
\newline
(ii) Let $\Sigma$ be the set of $L(P)$-sentences which are true of every $(M',A')$ where $M'$ is an $L$-structure, and $A'$ is a finite subset of the interpretation of the sort $X$in $M'$. Then $(M,A)\models \Sigma$.
\newline
(iii) $(M,A)$ is elementarily equivalent to some ultraproduct of $L(P)$-structures $(M',A')$ where $A'$ is finite. 

\end{Lemma}
We will now talk about the standard model $\mathbb V$ of set theory and saturated elementary extensions ${\mathbb V}^{*}$ of $\mathbb V$. It doesn't really make so much sense, but really we work with some small fragment of set theory including the natural numbers, the reals and all arithmetic operations on them together with cardinality maps for finite sets. The reader can work out for himself or herself the appropriate rigorous statements. 

\begin{Proposition} Suppose $(M,A)$ is pseudofinite. Then there is a (saturated if you wish) elementary extension ${\mathbb V}^{*}$ of $\mathbb V$ and some $(M^{*},A^{*})$ in ${\mathbb V}^{*}$ such that
\newline
(i) $(M^{*}, A^{*})$ is elementarily equivalent to $(M,A)$,
\newline
(ii) $A^{*}$ is finite in the sense of ${\mathbb V}^{*}$, and 
\newline
(iii)  Whenever $\psi$ is a formula of set theory which is true in ${\mathbb V}^{*}$ of $(M^{*},A^{*})$ then there is $(M,A)$ (in the standard model), such that $\psi$ is true of $(M,A)$ and $A$ is finite.
\newline
Moreover suppose that $(M,A)$ is a model of the common theory of $(M_{n}, A_{n})$ for $n<
\omega$ where $A_{n}$ is finite and of increasing size, and $A$ is infinite, then  $(M^{*},A^{*})$ can be chosen to satisfy also
\newline
(iii)' Whenever $\psi$ is a formula of set theory true of $(M^{*},A^{*})$ in ${\mathbb V}^{*}$ then $\psi$ is true of infinitely many $(M_{n},A_{n})$ in $\mathbb V$. 
\end{Proposition}
\begin{proof} This is a compactness argument. Consider the complete diagram of $\mathbb V$ together with  set of formulas $\psi(y,z)$ true of every $(M,A)$ in $\mathbb V$ where $M$ is an $L$-structure and $A$ a finite subset of the appropriate sort, as well as the formulas expressing that $(y,z)$ is elementarily equivalent (in $L(P)$) to $(M,A)$. It is finitely satisfiable (in ${\mathbb V}$), so has a (saturated if you wish) model. The moreover statement is also clear. 
\end{proof} 

\begin{Remark} Typically we take ${\mathbb V}^{*}$ to be saturated so $(M^{*},A^{*})$ will be appropriately saturated, so isomorphic to $(M,A)$ assuming the latter was already saturated (assuming some set theory and appropriate choices of degree of saturation).
\end{Remark}

Given ${\mathbb V}^{*}$ and $(M^{*},A^{*})$ as in Proposition 1.6, as $A^{*}$ is finite in the sense of ${\mathbb V}^{*}$, every internal subset $Z$ of $A^{*}$ has a finite cardinality in the sense of ${\mathbb V}^{*}$ (i.e. $|Z|\in\N^{*}$) and we obtain the nonstandard normalized counting measure $\mu^{*}$ on the Boolean algebra of internal subsets of $A^{*}$ which takes $Z$ to $|Z|/|A^{*}|$, a number in $[0,1]^{*}$.  For $Z$ a definable (in $M^{*}$) subset of the ambient sort $X$ in which $A^{*}$ lives, define $\mu^{*}(Z) = \mu^{*}(Z\cap A^{*})$. 
So $\mu^{*}$ gives a ``nonstandard" Keisler measure on the sort $X$ in $M^{*}$, in the sense that the values of $\mu^{*}$ are in the nonstandard unit interval (as well as finite additivity etc). We define $\mu$ to be the standard part of $\mu^{*}$ (restricted to definable sets) and we see that that $\mu$ is Keisler measure on the sort $X$ is the $L$-structure $M^{*}$, which we call the {\em pseudofinite} Keisler measure on $X$ given by $A^{*}$ (and the ambient structure ${\mathbb V}^{*}$).

The following is important (and well-known). It can be proved by an adaptation of the material in section 2.2 of \cite{CP}. In any case we follow the notation and context of Proposition 1.6 and the above construction. 

\begin{Proposition}  Assuming $Th(M^{*})$ is $NIP$, then the psedofinite Keisler measure on the sort $X$ is generically stable, nanely definable over and finitely satisfiable in some small model $M_{0}$. 
\end{Proposition}

\section{The distal case}
 The distal regularity theorem \cite{CS1}  is an attractive generalization of a result of Fox et al \cite{Fox} on a strong regularity theorem for semialgebraic graphs.

Our treatment here is related to that of Simon
\cite{Simon-distal}, but we make more explicit the connection with compact domination.

The relevant structural theorem concerns arbitrary smooth Keisler measures. Recall that a Keisler measure $\mu_{x}$ over $M$ is smooth if it has a unique extenson over any $N$ containing $M$. And a global Keisler measure $\mu_{x}$ is said to be smooth over a small submodel $M_{0}$ if $\mu$ is the unique extension over $\bar M$ of $\mu|M_{0}$. 

Here is the domination theorem for smooth measures.  It is more or less tautological. 

\begin{Proposition} Fix $T$, sort $X$, saturated model ${\bar M}$ and  Keisler measure $\mu$ on $X$ over ${\bar M}$. Let $M_{0}$ be a small elementary submodel of ${\bar M}$ and $\pi:X\to S_{X}(M_{0})$ the tautological map. Suppose $\mu$ is smooth over $M_{0}$.  Then for every definable (with parameters from ${\bar M}$) subset $Y$ of $X$ there is a closed subset $E$ of $S_{X}(M_{0})$ of $\mu_{0}$-measure $0$ such that for every $p\in S_{X}(M_{0})$ either $\pi^{-1}(p)\subset Y$ or $\pi^{-1}(p)\cap Y = \emptyset$. 
\end{Proposition}  
\begin{proof} 
 Otherwise the (closed) set $E$ of $p\in S_{X}(M_{0})$ such that $p(x)$ is consistent with each of  $x\in Y$ and $x\notin Y$, has $\mu_{0}$-measure $= \alpha > 0$. Let $(\mu_{0})_{E}$ be the localization of $\mu_{0}$ to $E$. Then $(\mu_{0})_{E}$ has two different extensions to a measure over ${\bar M}$, one giving $Y$ measure $1$ and one giving it measure $0$.  It follows that $\mu_{0}$ itself has two different extensions to $\bar M$, contradicting smoothness. 
\end{proof}

The following strong regularity  (or Ramsey-type) statement is a simple compactness argument applied to Proposition 2.1.

\begin{Corollary} Let $(V,W, R)$ be a bipartite graph definable in a structure $M$. Let $\mu$ be a smooth Keisler measure on $V$ over $M$, and $\nu$ an arbitrary Keisler measure on $W$ over $M$. Let $\epsilon > 0$. Then there are partitions $V = V_{1}\cup ..\cup V_{n}$ and $W = W_{1}\cup ...\cup W_{m}$ of $V$, $W$ respectively into definable sets, and a set $\Sigma$ of pairs $(i,j)$ of indices such that
\newline
(i) $(\mu\times \nu )(\cup_{(i,j)\in \Sigma}(V_{i}\times V_{j})) < \epsilon$ and 
\newline
(ii)  for $(i,j)\notin \Sigma$, $V_{i}\times W_{j}$ is homogeneous for $R$, namely $V_{i}\times W_{j}$ is either contained in $R$ or disjoint from $R$. 

\end{Corollary} 

\begin{proof}  We may assume $M$ to be a saturated. Let $M_{0}$ be a small elementary submodel of $M$ such that $\mu$ is smooth over $M_{0}$ and $R(x,y)$ is definable over $M_{0}$. We make use of Proposition 2.1 with $X = V$. 

Fix $\epsilon > 0$. For any $b$,  Let $E_{b}$ the closed $\mu_{0}$-measure $0$ subset of $S_{V}(M_{0})$ outside of which each fibre of $\pi$  is either contained in or disjoint from $R(x,b)$. Clearly $E_{b}$ depends only on $tp(b/M_{0})$ and so we write as $E_{q}$.  Let $Z_{q}$ be an $M_{0}$-definable set containing $E_{q}$ and with $\mu_{0}$-measure $<\epsilon$.  By compactness we can partition $V\setminus Z_{q}$ into $M_{0}$-definable sets $V_{q,1},...,V_{q,n_{q}}$ such that for each $i$, $\pi^{-1}(V_{q,i})$ is either contained in $R(x,b)$ (for some/all $b$ realizing $q$) or disjoint from $R(x.b)$ (for some/all $b$ realizing $q$). We can now, by compactness,  replace $q$ by a formula (or $M_{0}$-definable set) $W_{q}$ in $q$, so that for each $i\leq n_{q}$ either $V_{q,i}$ is contained in $R(x,b)$ for all $b\in W_{q}$,  or $V_{q,i}$ is disjoint from $R(x,b)$ for all $b\in W_{q}$.

Doing this for each $q$ and applying compactness gives us a partition $W_{q_{1}},...W_{q_{m}}$ of $W$ into $M_{0}$-definable sets, and for each $j=1,..,m$ a partition $V = V_{q_{j},1} \cup V_{q_{j},2} \cup ...\cup V_{q_{j},n_{q_{j}}}\cup Z_{q_{j}}$, such that $\mu_{0}(Z_{q_{j}}) = 0$ for all $j$,  and for all $j,i$, 
\newline
(*)  $\pi^{-1}(V_{q_{j},i})$ is either contained in $R(x,b)$ for all $b\in W_{q_{j}}$  or is disjoint from $R(x,b)$ for all $b\in W_{q_{j}}$.

\vspace{2mm}
\noindent
Let $V_{1},...,V_{t}$ be a common refinement of this finite collection of partitions of $V$.  And we claim that this partition, together with the partition $W_{q_{1}},..,W_{q_{m}}$ of $W$ is as required. We have to identify the exceptional set $E$ of pairs. So let $E = \{(i,q_{j}):V_{i}\subseteq Z_{q_{j}}\}$.  For each $q_{j}$, $\cup \{V_{i}\times W_{q_{j}}:V_{i}\subseteq Z_{q_{j}}\} = Z_{q_{j}}\times W_{q_{j}}$ which has $\mu\times \nu$ measure $<\epsilon\nu(W_{q_{j}})$. So summing over the $q_{j}$ we get $(\mu\times\nu)(\cup_{(i,q_{j})\in E}(V_{i}\times W_{q_{j}})) < \epsilon$.
And for $(i,q_{j})\notin E$, $V_{i}$ will be contained in $V_{q_{j},i}$ for some $i$, so by (*) $V_{i}\times W_{q_{j}}$ is either contained in or disjoint from $R$. 
\end{proof}

The notion of a {\em distal} first order theory $T$ was introduced by Pierre Simon in his thesis (see \cite{Simon-book}). One of the characterizations of distality is that $T$ has $NIP$ and every generically stable measure is smooth.  Among distal theories are  $o$-minimal theories (such as $RCF$), the theory of $\Q_{p}$, and $Th(\Z,+,<)$.

So here is the distal regularity theorem, stated for suitable families of finite graphs. 
\begin{Proposition}  Let ${\cal G} = (G_{i}:i\in I)$ be a family of finite (bipartiite) graphs $G = (V,W, R)$ such that one of the following happens:
\newline
(i) The graphs are uniformly definable in some model $M$ of a distal theory $T$,
\newline
(ii) For some model $M$ of some distal theory $T$, there is a graph $(V,W,R)$ definable in $M$ such that ${\cal G}$ is the family of finite (induced) subgraphs of $(V,W, E)$,
 or
\newline
(iii)  Every model $(V,W,R)$ of the common theory of the $G_{i}$'s is interpretable in a model of some distal theory.
\newline
THEN, for any $\epsilon$ there is $N_{\epsilon}$, such that for every $(V,W,R)\in {\cal G}$, there are partitions $V = V_{1}\cup .. \cup V_{n}$, and $W = W_{1} \cup .. \cup W_{m}$, with $n, m < N_{\epsilon}$ such that for some some ``exceptional" set $E$ of pairs $(i,j)$ (with $1\leq i \leq n$ and $1\leq j \leq m$),
\newline
the cardinality of $\cup_{(i,j)\in E}(V_{i}\times W_{j})$ is $<\epsilon|V||W|$, and for all $(i,j)\notin E$, $V_{i}\times W_{j}$ is either contained in $R$ or disjoint from $R$.
\end{Proposition}
\begin{proof}  Context (ii) is the one dealt with in \cite{CS1} and which generalizes \cite{Fox}.   Note that Context (i) would be vacuous when $T$ is $o$-minimal as we have finite bounds on the cardinalities of uniformly definable finite sets, but for the $p$-adics and/or Presburger, it is nonvacuous. 

The proof of the proposition uses Propositions 1.6 and 1.8 (with a possible variant in Context (iii)).
 We focus here on Context (i).  Suppose the conclusion fails. So for some fixed $\epsilon$, no finite $N$ works. So we can find  $N_{1} < N_{2} < ....$, and counterexamples $G_{N_{1}}$ for $N_{1}$, $G_{N_{2}}$ for $N_{2}$ , etc.  in ${\cal G}$ with the cardinalities of the vertex sets increasing.  Proposition 1.6 gives us a (saturated)  model $M^{*}$ of $T$ in some ${\mathbb V}^{*}$ and definable $G^{*} = (V^{*},W^{*},R^{*})$ in $M^{*}$ such that $V^{*}$, $W^{*}$ are finite in the sense on ${\mathbb V}^{*}$ and the moreover clause (iii)' holds.
Let $\mu^{*}$, $\nu^{*}$  be the nonstandard normalized counting measures on $V^{*}$ and $W^{*}$,  given by the construction following Remark 1.7, and let $\mu$ and  $\nu$\,  be the corresponding pseudofinite Keisler measures. By 1.8 $\mu$ is generically stable, so smooth as $Th(M^{*})$ is distal. 
Apply Corollary 2.2 to $(V^{*},W^{*},R^{*})$ with say $\epsilon/2$, to get (definable) partitions $V_{1}^{*},..,V_{n}^{*}$ of $V$ and $W_{1}^{*},..,W_{m}^{*}$ of $W$, and an exceptional set $E$ of pairs $(i,j)$ satisfying the conclusions of 2.2.  Choose $\epsilon/2 < \delta < \epsilon$, and we can express the existence of the partitions and that $(\mu^{*}\times \nu^{*})(\cup_{(i.j)\in E}(V_{i}^{*}\times W_{j}^{*})) < \delta$ and that for ($(i,j)\notin E$, $V_{i}^{*}\times W_{j}^{*}$ is homogeneous for $R^{*}$ by the truth of formula $\psi$ of set theory for $M^{*},G^{*}$ in ${\mathbb V}^{*}$.  The moreover clause of Proposition 1.6 tells that $\psi$ is true for infinitely many of the $(M,G_{N_{k}})$ in $\mathbb V$, and for $N_{k} > n,m$ we get a contradiction. 

\end{proof}

\section{The NIP case}
The regularity lemma for finite graphs $(V,W, R)$ where the relation $R$ is ``$k-NIP$"  (or $VC$-dimension bounded by approximately $k$) has a nice and elementary direct proof in \cite{CS2} (in the greater generality of hypergraphs). However we want to again deduce it just from a domination statement, so we work in the context where such statements are currently available, namely inside a $NIP$ theory.

We first recall the notion ``$\mu$-wide". If $\mu_{x}$ is a (say global) Keisler measure and $\Sigma(x)$ is a partial type over a small set, we say that $\Sigma(x)$ is $\mu$-wide if every finite conjunction of formulas in $\Sigma$ has $\mu$ measure $>0$.  

Again we start with the (generic) domination theorem for generically stable measures.

\begin{Proposition}
 Suppose $T$ is $NIP$. Let $\mu$ be a global generically stable measure on the definable set (or sort) $X$ and assume $\mu$ does not fork over $M_{0}$ (so is definable over $M_{0}$). 
Let $\pi: X \to S_{X}(M_{0})$ be the tautological map, and let $\mu_{0}$ be the induced measure on $S_{X}(M_{0})$. Then for every definable (with parameters from $\bar M$)  subset $Y$ of $X$ there is a closed set $E\subseteq S_{X}(M_{0})$ of $\mu_{0}$ measure $0$ such that for each $p\in S_{X}(M_{0})\setminus E$, not both $p\cup ``x\in Y"$ and $p\cup``x\notin Y$ are $\mu$-wide. 
\end{Proposition}
\begin{proof}  We deduce this formally from the basic results in \cite{NIPIII}.  First by Proposition 3.3 of \cite{NIPIII}, $\mu$ is the the unique global nonforking extension of its restriction to $M_{0}$. 
Let ${\cal P}$ the space of global complete types $p(x)$ which do not fork over $M_{0}$, and let $\pi'$ be the restriction map from ${\cal P}$ to  $S_{X}(M_{0})$. Then by Theorem 5.4, of \cite{NIPIII}, ${\cal P}$ is dominated by $((S_{X}(M_{0}),\pi',\mu)$ in the sense that for any formula $\phi(x)$ over ${\bar M}$ the set $E$ of $p\in S_{X}(M_{0})$ such that $\pi'^{-1}(p)$ intersects both (the clopen determined by) $\phi(x)$ and (the clopen determined by $\neg\phi(x)$, has $\mu_{0}$ measure $0$. Note that $E$ is closed. Recall that we  $\pi:X\to S_{X}(M_{0})$ takes $a\in X({\bar M})$ to $tp(a/M_{0})$.  Now suppose that $p(x)\in S_{X}(M_{0})\setminus E$.  If $p(x)\cup\{\phi(x)\}$ is $\mu$-wide, then as $\mu$ does not fork over $M_{0}$, $p(x)\cup\{\phi(x)\}$ does not fork over $M_{0}$ so extends to some $p'\in {\cal P}$. Likewise if $p(x)\cup\{\neg\phi(x)\}$ is $\mu$-wide, it extends to some $p''\in {\cal P}$. So  as $p\notin E$, not both can happen. 
\end{proof}

A simple compactness argument applied to Proposition 3.1 again gives a strong regularity theorem. 

\begin{Corollary} Suppose $T$ is $NIP$ and $(V,W, R)$ is a graph definable in a model $M$ of $T$. Let $\mu$ be a generically stable measure on $V$ over $M$ and $\nu$ any Keisler measure on $W$ over $M$. Fix $\epsilon > 0$. Then there are partitions $V = V_{1}\cup....\cup V_{n}$ and $W = W_{1}\cup ..\cup W_{m}$ of $V, W$ into definable sets, and an exceptional set $E$ of pairs $(i,j)$ of indices such that
\newline
(i) The $\mu\times \nu$ measure of $\cup_{(i,j)\in E}V_{i}\times W_{j}$ is $< \epsilon$, and
\newline
(ii)  For any $(i,j)\notin E$, either $(\mu\otimes\nu)((V_{i}\times W_{j})\cap R) <\epsilon\mu(V_{i})(\nu(W_{j})$ or  
$(\mu\otimes\nu)((V_{i}\times W_{j})\setminus R) <\epsilon\mu(V_{i})(\nu(W_{j})$
\end{Corollary}

\begin{proof}  We follow the proof of Corollary 2.2, but with $\epsilon$-homogeneous in place of homogeneous (using definability over $M_{0}$ of $\mu$), and paying slightly more attention to the exceptional set. 

Again assume $M$ to be saturated, and suppose $\mu$ does not fork over $M_{0}$.  Fix $\epsilon > 0$. We use Proposition 3.1 with $X = V$.  For each $q\in S_{W}(M_{0})$ we find closed $E_{q}\subseteq S_{V}(M_{0})$ of $\mu_{0}$-measure $0$, such that for each $p\in S_{V}(M_{0})\setminus E_{q}$ and some (any) $b$ realizing $q$ at most one of $p(x)\cup\{R(x,b)\}$, $p(x)\cup\{\neg R(x,b)\}$ is $\mu$-wide.  Let $Z_{q}$ be an $M_{0}$-definable set containg $E_{q}$ and of $\mu_{0}$-measure $<\epsilon/2$.  By compactness we can partition $V\setminus E_{q}$ into $M_{0}$-definable sets $V_{q_{1}},..,V_{q,n_{q}}$ such that for each $i$, either $\mu(V_{q,i}\cap R(x,b)) = 0$ (for some/all $b$ realizing $q$), or $\mu(V_{q_{i}}\setminus R(x,b)) = 0$ for some/all $b$ realizing $q$).  We may assume that $\mu(V_{q,i})) > 0$ for each $i$ (otherwise just add it to $Z_{q}$).

Now we use definability of $\mu$ over $M_{0}$ to find an $M_{0}$-definable set $W_{q}$ containing $q$ such that for each $i=1,...,n_{q}$ exactly one of the following holds:

$(i)_{q}$ forall $b\in W_{q}$, $\mu(V_{q,i}\cap R(x,b)) < (\epsilon^{2}/2)\mu(V_{q,i})$.
\newline
$(ii)_{q}$  for all $b\in W_{q}$, $\mu(V_{q,i}\setminus R(x,b)) <(\epsilon^{2}/2)\mu(V_{q,i})$. 

By compactness we can find $q_{1},..,q_{m}$ such that $W_{q_{1}},...,W_{q_{m}}$  partition $W$.  Again we find a common refinement $V_{1},..,V_{r}$ of the finitely many partitions $V_{q_{j},1}, ...,  V_{q_{j},n_{j}}, Z_{q_{j}}$ of $V$.

Then $V = V_{1}\cup ...\cup V_{r}$ and $W = W_{q_{1}}\cup ..\cup W_{q_{m}}$ will be the desired partitions. We have to check that it works. 

We have to  identify  the exceptional set of pairs of indices.

To that avail let us fix some $q_{i}$ and call it $q$, and we focus on the subgraph $(V,W_{q}, R|(V\times W_{q}))$.  Let $I = \{i: 1\leq i\leq n_{q}:$ and $(i)_{q}$  above holds\}. Let $J$ be the rest of the indices $i$ between $1$ and $n_{q}$, namely where $(ii)_{q}$ holds. 

Let $B\subseteq V\times W_{q}$ be $\cup_{i\in I}((V_{q,i}\times W_{q})\cap R)  \cup\cup_{i\in J}((V_{q,i}\times W_{q})\setminus R)$. it is then clear that
\newline
{\em Claim 1.}  $(\mu\otimes \nu) (B) < (\epsilon^{2}/2)\nu(W_{q})$.

Let $\Sigma_{q,1}$ be the set of indices $i=1,..,r$ such that  $V_{i}\subseteq Z_{q}$, and $\Sigma_{q,2}$ the set of indices $i$ such that  $(\mu\otimes\nu)((V_{i}\times W_{q}) \cap B)\geq\epsilon\mu(V_{i})\nu(W_{q})$.  Let $\Sigma_{q}$ be the (disjoint) union of $\Sigma_{q,1}$ and $\Sigma_{q,2}$. 

Then 
\newline
{\em Claim 2.} $\sum_{i\in\Sigma_{q}}(\mu\otimes\nu) (V_{i}\times W_{q}) <\epsilon\nu(W_{q})$.
\newline
{\em Proof of Claim 2.} Note that $\sum_{i\in \Sigma_{q,1}}(\mu\otimes\nu)(V_{i}\times W_{q}) \leq (\mu\otimes \nu)(Z_{q}\times W_{q}) < (\epsilon/2)\nu(W_{q})$. So it suffices to prove that $\sum_{i\in\Sigma_{q,2}}(\mu\otimes\nu) (V_{i}\times W_{q}) <(\epsilon/2)\nu(W_{q})$. If not then by the definition of $\Sigma_{q,2}$, $(\mu\otimes \nu)(B) \geq \sum_{i\in \Sigma_{q,2}}(\mu\otimes\nu)((V_{i}\times W_{q})\cap B)\geq \sum_{i\in \Sigma_{q,2}}\epsilon\mu(V_{i})\nu(W_{q}) = \epsilon\sum_{i\in\Sigma_{q,2}}(\mu\otimes\nu)(V_{i}\times W_{q})\geq  (\epsilon^{2}/2)\nu(W_{q})$, which contradicts Claim 1.

Now suppose $t\notin\Sigma_{q}$. 
\newline
{\em Case (i).}  $V_{t}\subseteq  V_{q,i}$ for some $i\in I$. 
\newline
So $(V_{t}\times W_{q})\cap B  = (V_{t}\cap W_{q})\cap R$ and  has $\mu\otimes\nu$ measure $<\epsilon\mu(V_{t})\nu(W_{q})$. 

\vspace{2mm}
\noindent
{\em Case (ii).}  $V_{t}\subseteq V_{q,i}$ for some $i\in J$. 
\newline
Likewise we have that $(\mu\otimes\nu)((V_{t}\times W_{q})\setminus R) <\epsilon\mu(V_{t})\nu(W_{q})$. 

\vspace{2mm}
\noindent
Now let the global exceptional set $E = \{(i,q_{j}): i\in \Sigma_{q_{j}}: i=1,..,r, j=1,..,m\}$, and we see from Claim 2 (as well as the Case (i), Case (ii) discussion above) that that the conclusions (i) and (ii) or Corollary 3.2 are satisfied. 

\end{proof}

The application to families of finite graphs is almost identical to Proposition 2.3, with a similar proof, but we state it anyway.
\begin{Proposition}  Let ${\cal G} = (G_{i}:i\in I)$ be a family of finite (bipartiite) graphs $G = (V,W, R)$ such that one of the following happens:
\newline
(i) The graphs are uniformly definable in some model $M$ of an $NIP$ theory $T$. ,
\newline
(ii) For some model $M$ of some $NIP$ theory $T$, there is a graph $(V,W,R)$ definable in $M$ such that ${\cal G}$ is the family of finite (induced) subgraphs of $(V,W, E)$,
 or
\newline
(iii)  Every model $(V,W,R)$ of the common theory of the $G_{i}$'s is interpretable in a model of some $NIP$ theory.
\newline
THEN, for any $\epsilon$ there is $N_{\epsilon}$, such that for every $(V,W,R)\in {\cal G}$, there are partitions $V = V_{1}\cup .. \cup V_{n}$, and $W = W_{1} \cup .. \cup W_{m}$, with $n, m < N_{\epsilon}$ such that for some some ``exceptional" set $E$ of pairs $(i,j)$ (with $1\leq i \leq n$ and $1\leq j \leq m$),
\newline
(a)  cardinality of $\cup_{(i,j)\in E}V_{i}\times W_{j}$ is $<\epsilon|V||W|$, and 
\newline
(b) for all $(i,j)\notin E$, either $|(V_{i}\times W_{j})\cap R| < \epsilon|V_{i}||W_{j}|$ or $|(V_{i}\times W_{j})\setminus R| < \epsilon|V_{i}||W_{j}|$
\end{Proposition}

\begin{Remark} (i) Note that Corollary 3.2 depends only on the Keisler measure $\mu$ satisfying the generic domination statement over $M_{0}$ in Proposition 3.1, as well as being definable over $M_{0}$. (i.e. full generic stability of $\mu$ and $NIP$-ness of $T$ are not needed). 
\newline
(ii)  Likewise, if $R(x,y)$ is $\phi(x,y)$, assuming the domination statement and definability for a $\phi$-measure $\mu$, Corollary 3.2 will hold.

\end{Remark}

\section{The stable case} 

The stable regularity theorem concerns finite graphs $(V,W, R)$ where the edge relation $R(x,y)$ is $k$-stable, and gives a partition into almost homogeneous subgraphs but without any exceptional set.  The original statement and proof are in \cite{Malliaris-Shelah} and involve finite combinatorics in the presence of the Shelah $2$-rank and give optimal bounds. A pseudofinite proof making use of local stability theory was given in \cite{Malliaris-Pillay}. The proof we present here is a  simplification of the latter. There is no explicit use of any local ranks, other than   ingredients in the proof of Fact 1.1.

We first discuss the methods and relationship with the previous proofs.  Fix a complete theory $T$ and a stable formula $\phi(x,y)$, where $x$ is of sort $X$.  Let $\mu$ be a 
$\phi$- measure on $X$ over a saturated model say $\bar M$. Let $M_{0}$ be a small model such that $\mu$ does not fork over $M_{0}$ (i.e. every $\phi$ formula over $M_{0}$ with positive measure does not fork over $M_{0}$). Now every complete $\phi$-type $p$ over $M_{0}$ has a unique 
nonforking extension over $\bar M$ (i.e. to a complete $\phi$-type over $\bar M$). It follows that for each $p\in S_{\phi}(M_{0})$ and any $b\in {\bar M}$ at most one of $p\cup\{\phi(x,b)\}$, $p\cup\{\neg\phi(x,b)\}$ is $\mu$-wide. So we have domination of $X$ by $S_{\phi}(M_{0})$ (but with no exceptional sets).  So we can run the proof of Corollary 3.2, but note that it nevertheless gives a possibly nonempty set of exceptional pairs $(i,j)$ in the regularity statement.  So more is needed, and this is precisely Fact 1.1.

Remember that a bipartite graph  $(V,W,R)$  (or rather the edge relation $R$ on this graph)  is $k$-stable if there do not exist $a_{1},,..,a_{k}\in V$, $b_{1},..,b_{k}\in W$ such that $R(a_{i},b_{j})$ iff $i\leq j$.   Given an $L$-structure $M$ and $L$-formula $\phi(x,y)$ we get a corresponding bipartite graph $(X,Y, R)$  (where $X$ is the $x$-sort in $M$, $Y$ the $y$-sort in $M$ and $R$ the interpretation of $\phi(x,y)$ in $M$). And the formula $\phi(x,y)$ is stable for $Th(M)$ iff $(X,Y,R)$ is $k$-stable for some finite $k$. 
$\phi^{*}$ is the same formula as $\phi$ except the roles of variable variable and parameter variable are interchanged.

We first give the strong regularity theorem (analogues of Corollaries 2.2 and 3.2) for arbitrary graphs where the edge relation is stable.  
\begin{Proposition} Let $(V,W,R)$ be a graph definable in some structure $M$ where the relation $R(x,y)$ is stable.  Identify $R$ with the $L$-formula defining it.   Let $\mu_{x}$ be a Keisler measure on $V$ over $M$,  Then for each $\epsilon > 0$ there are partitions $V = V_{1}\cup ...\cup V_{n}$ of $V$ and $W = W_{1}\cup ..\cup W_{m}$ of $W$ into definable sets such for each $i,j$,
either for all $b\in W_{j}$, $\mu(V_{i}\setminus R(x,b)) \leq \epsilon\mu(V_{i})$, or for all $b\in W_{j}$, $\mu(V_{i}\cap R(x,b)) \leq\epsilon\mu(V_{i})$. 
Moreover each $V_{i}$ can defined by an $R$-formula, and each $W_{j}$ by a $R^*$-formula.

\end{Proposition}

\begin{proof} There is no harm in assuming the language to be countable, and $M = \bar M$ to be saturated.  By Fact 1.1, $\mu|R = \sum_{i\in I}\alpha_{i}p_{i}'$, for some countable $I$, complete global $R$-types $p_{i}'$, and  $\alpha_{i}$ with $0<\alpha_{i}\leq 1$ such that $\sum_{i}\alpha_{i} = 1$. 

So note that $\mu(p_{i}') = \alpha_{i} >0$  for $i\in I$. 

Let $M_{0}$ be a countable model such that $\mu$ does not fork over $M_{0}$, equivalently each $p_{i}'$, for $i\in I$ does not fork over $M_{0}$.
Let $p_{i}$ be the restriction of $p_{i}'$ to $M_{0}$. So $p_{i}'$ is the unique global nonforking extension of $p_{i}$, and $\mu|R$ is the unique nonforking extension of 
the $R$-measure over $M_{0}$, $\sum_{i\in I}\alpha_{i}p_{i}$. Write $\mu_{0}$ for $\mu|M_{0}$.  Notice $\mu_{0}(p_{i}) = \alpha_{i}$. 

By stability,  $S_{R}(M_{0})$ is countable. Let $B$ be the countable (so Borel)  set $S_{R}(M_{0}\setminus\{p_{i}:i\in I\}$. Then $\mu_{0}(B) = 0$.

Now fix $\epsilon > 0$.  Let $U$ be a $R$-definable over $M_{0}$ set which contains $B$, and has $\mu_{0}$-measure $<\epsilon$. 
For each $i\in I$ such that $p_{i}\notin U$, let $V_{i}$ be a formula (or definable set) in $p_{i}$ such that $\mu(V_{i})<\alpha_{i}/(1-\epsilon)$. 
By compactness finitely many $V_{i}$, say $V_{1},..,V_{k}$ cover $U^{c}$ (complement of $U$ in $S_{R}(M_{0})$), and we may assume that the $V_{i}$ are disjoint. 
Let $\delta = (\alpha_{1}/(1-\epsilon)) - \mu(V_{1})$.  And let $U_{1}$ be a $R$-definable set  over $M_{0}$ set such that   $B\subseteq U_{1}\subset U$, and $\mu(U_{1}) <\delta$. 

As before find   $V_{k+1},...,V_{n}$  partitioning $U\setminus U_{1}$ (assuming the latter is nonempty), with $V_{j}\in p_{j}$  ($j\in I$) and 
$\mu(V_{j}) \leq \alpha_{j}/(1-\epsilon)$  (where again $\alpha_{j}= \mu_{0}(p_j) > 0$). 
Replace $V_{1}$ by $V_{1}\cup U_{1}$ and we have for the new $V_{1}$, $\mu(V_{1})\leq \alpha_{1}(1-\epsilon)$. 

So to summarise, we have so far:
\newline
{\em Claim.} $V_{1},...,V_{n}$ are $R$-definable over $M_{0}$ sets which partition $V$, and for each $i=1,..,n$ there is $p_{i}\in S_{R}(M_{0})$ with $\mu(p_{i}) > 0$, $V_{i}\in p_{i}$ and $\mu(V_{i}\setminus p_{i}) \leq \epsilon\mu(V_{i})$. 

\vspace{2mm}
\noindent
Now we  partition $W$ using the $R$-definitions of $p_{1},..,p_{n}$: 
For each $i=1,..,n$ the $R$-definition of $p_{i}$ (or equivalently $p_{i}'$) is a $R^{*}$-formula $\psi_{i}(y)$ over $M_{0}$ (with the property that for all $b\in {\bar M}$, $\phi(x,b)\in p_{i}'$ iff $\models \psi_{i}(b)$).  For each subset $I$ of $\{1,..,n\}$, let $W_{I}$ be the set defined by  $\wedge_{i\in I}\psi_{i}(y)\wedge\wedge_{i\notin I}\neg\psi_{i}(y)$.  So the $W_{I}$ partition $W$ into $R^{*}$-definable sets.  $V = V_{1}\cup ...\cup V_{n}$ and $W = \cup_{I}W_{I}$ will be the desired partitions of $V, W$. 

We have to check that the conclusions hold. 
Note that for for each $i\in \{1,..,n\}$ and $I\subseteq \{1,..,n\}$, we have either
\newline
(a)  for all $b\in W_{I}$, $\phi(x,b)\in p_{i}'$, or
\newline
(b) for all $b\in W_{I}$, $\neg\phi(x,b)\in p_{i}'$. 

In case (a), as $p_{i}'$ is the unique nonforking extension of $p_{i}$, for each $b\in W_{I}$, $p_{i}\cup\{\neg\phi(x,b)\}$  forks over $M_{0}$ so has $\mu$-measure $0$. Hence as $\mu(V_{i}\setminus p_{i})\leq\epsilon\mu(V_{i})$ it follows that  $\mu(V_{i}\setminus R(x,b))\leq\epsilon \mu(V_{i})$ for all $b\in W_{q}$.

Likewise in case (b), $\mu(V_{i}\cap R(x,b))\leq\epsilon\mu(V_{i})$ for all $b\in W_{I}$.

This completes the proof.

\end{proof} 
The stable regularity lemma (or a suitable version)  for families of finite graphs now follows as earlier:
\begin{Corollary} Fix $k$ and let ${\cal G}$ be the family of finite graphs $(V,W,R)$ where the relation $R$ is $k$-stable.  Then for any $\epsilon>0$, there is $N$ such that for each 
$(V,W,R)\in {\cal G}$ there are partitions  $V = V_{1}\cup ..\cup V_{n}$ and $W = W_{1}\cup ...\cup W_{m}$ with $n,m<N$ such that for each $i$ and $j$, either 
$|(V_{i}\times W_{j})\cap R| \leq \epsilon |V_{i}||W_{j}$ (so the induced graph on $V_{i},W_{j}$ is almost empty) or $|(V_{i}\times W_{j})\setminus R| \leq \epsilon|V_{i}||W_{j}|$ (so the induced graph on $V_{i}, W_{j}$ is almost complete).

\end{Corollary}

\end{document}